\documentclass[12pt]{amsart}

\usepackage[utf8]{inputenc}

\usepackage[english]{babel}
\usepackage{hyperref}
\usepackage[utf8]{inputenc}
\usepackage{graphicx}
\usepackage{graphics}
\usepackage{amssymb}
\usepackage{amsmath}
\usepackage{amsthm}
\usepackage{tikz-cd}
\usepackage{mathrsfs}
\usepackage[colorinlistoftodos]{todonotes}
\usepackage{enumitem}
\usepackage{lmodern}
\usepackage[T1]{fontenc}
\usepackage{import}
\usepackage{wrapfig}
\usepackage{calc}
\usepackage[all]{xy}
\usepackage[margin=3.5cm]{geometry}
\usepackage{float}
\usepackage{mathtools}
\usepackage{subcaption}

\newtheorem{thm}{Theorem}[section]
\newtheorem{theorem}{Theorem}[section]

\newtheorem{prop}[thm]{Proposition}
\newtheorem{rmk}[thm]{Remark}

\def\R{\mathbb{R}}
\def\ep{\epsilon}
\def\C{\mathbb{C}}

\newtheorem{question}{Question}

\newtheorem{remark}[theorem]{Remark}
\newtheorem*{theorem-non}{Theorem}

\def\ep{\epsilon}

\def\R{\mathbb{R}}

\def\N{\mathbb{N}}

\def\C{\mathbb{C}}

\def\cal R{\mathcal R}

\newcommand{\jznote}[1]{#1}

\newcommand{\jgnote}[1]{#1}

\newcommand{\djnote}[1]{#1}

\newcommand{\zjnote}[1]{#1}

\newcommand{\zjrnote}[1]{#1}

\title{Coarse distance from \zjnote{dynamically convex to convex}}

\author{Julien Dardennes}
\email{julien.dardennes@math.univ-toulouse.fr}
\address{Institute of Mathematics, University of Toulouse, 118 route de Narbonne, 31062 Toulouse Cedex 9, France}

\author{Jean Gutt}
\email{jean.gutt@math.univ-toulouse.fr}
\address{Institute of Mathematics, University of Toulouse,  118 route de Narbonne, 31062 Toulouse Cedex 9, France and INU Champollion, Place de Verdun, 81000 Albi, France}

\author{Vinicius G. B. Ramos}
\email{vgbramos@impa.br}
\address{Instituto de Matem\'atica Pura e Aplicada, Estrada Dona Castorina, 110, Rio de Janeiro - RJ -Brasil, 22460-320}

\author{Jun Zhang}
\email{jzhang4518@ustc.edu.cn}
\address{The Institute of Geometry and Physics, University of Science and Technology of China, 96 Jinzhai Road, Hefei Anhui, 230026, China}

\begin{document}

\maketitle

\begin{abstract}

Chaidez and Edtmair have recently found the first examples of dynamically convex domains in $\R^4$ that are not symplectomorphic to convex domains \zjnote{(called symplectically convex domains)}, answering a long-standing open question. \djnote{In this paper, we \zjnote{discover} new examples of such domains without referring to Chaidez-Edtmair's criterion in \cite{CD-20}. We also show that these domains are arbitrarily far from the set of \zjnote{symplectically} convex domains in $\R^4$ with respect to the \zjnote{coarse symplectic Banach-Mazur distance} by using an explicit numerical criterion for symplectic non-convexity.
}

\end{abstract}


\section{Introduction and main results}  \label{sec-motivation} \subsection{Introduction} 
Convex domains in $\R^{2n}$ have very strong symplectic rigidity properties, \jznote{including the existence of closed orbits on its boundary in many cases \cite{Rab78,HWZ}, the celebrated strong Viterbo conjecture that claims the equality of all the normalized symplectic capacities, as well as its interesting consequences on the systolic inequality between symplectic capacity and volume \cite{Viterbo2000} and its implication to Mahler conjecture in convex geometry \cite{AA-K-O-2014}. However,} geometric convexity is not preserved by symplectic transformations. 

The main motivation for this paper is the desire to understand {\bf the symplectic analogue of convexity}.
Hofer, Wysocki and Zehnder observed that the characteristic flow on the boundary of a convex set satisfies a certain index condition \cite{HWZ}, that they called {\em dynamical convexity}, which is preserved by symplectomorphisms. For a long time, it was speculated that if a domain is dynamically convex, then it is symplectomorphic to a convex domain. \jznote{Chaidez and Edtmair have shown \cite{CD-20} that there exist dynamically convex domains that are not symplectomorphic to convex ones} \djnote{in dimension 4}. 
\djnote{A recent work in \cite{DGZ} discovers more such examples which have been then easily applied to Chaidez and Edtmair's generalization of their criterion in higher dimensions \cite{CD-higher}.}

\medskip

In this paper we answer the following \djnote{quantitative} question in dimension 4.
\begin{question} \label{question-1}
	\jznote{How far can dynamically convex \jznote{domains} be away from convex domains?}
\end{question}


%

Let $\mathcal C_{4}$ be the set of domains in $\R^{4}$ which are \zjnote{symplectomorphic to convex domains, called symplectically convex,} and let \jznote{$\mathcal{D}_4$} denote the set of dynamically convex domains as \jznote{first introduced in \cite{HWZ}}.
It follows from \djnote{the same paper} that $\mathcal{C}_4\subset\mathcal{D}_4$.

A non-linear symplectic analogue of the Banach-Mazur distance (\ref{dfn-bm})  was \jznote{first suggested by Ostrover and Polterovich and further developed by \cite{Usher-sbm,SZ}. It is defined as follows.} 
For $U,V\subset\R^4$, let
\begin{equation} \label{dfn-dc} 
d_c(U, V) = \inf\left\{\log \lambda \geq 0 \, \bigg| \, \frac{1}{\lambda} U \hookrightarrow V \hookrightarrow \lambda U\right\}
\end{equation}
where $\hookrightarrow$ represents a symplectic embedding via some symplectomorphism $\phi$ of \zjnote{$\R^{4}$. According to \cite[Definition 1.3]{Usher-sbm} where $d_c$ is named as the coarse symplectic Banach-Mazur distance, we briefly call $d_c$ the {\it coarse distance} in this paper}. Following what is usually done in any metric space, given $\mathcal{A},\mathcal{B}$ two collections of subsets of $\R^4$, we define the distance from $\mathcal{A}$ to $\mathcal{B}$ as: 
\jznote{
\begin{equation} \label{sigma}
\sigma_{d_c}(\mathcal{A},\mathcal{B}) = \sup_{U \in \mathcal{A}} \,\inf_{V \in \mathcal{B}} \,d_c(U,V).
\end{equation}}
Note that \jznote{$\sigma_{d_c}(-, -)$} is not symmetric in general. For instance, if $\mathcal{A} \subset \mathcal{B}$, then \jznote{$\sigma_{d_c}(\mathcal{A},\mathcal{B}) = 0$} while \jznote{$\sigma_{d_c}(\mathcal{B},\mathcal{A})$} could be large. Moreover, \jznote{$\sigma_{d_c}(-,-)$} satisfies a monotonicity property, namely, \jznote{$\sigma_{d_c}(\mathcal{A}, \mathcal{B}') \leq \sigma_{d_c}(\mathcal{A}, \mathcal{B})$} if $\mathcal{B} \subset \mathcal{B}'$. \jznote{In what follows, for simplicity, let us denote $\sigma_{d_c}$ by $\sigma$. }

\medskip

The main result of this paper is \jznote{to answer Question \ref{question-1} under the distance $d_c$}. It turns out that \jznote{the $d_c$-distance} from the set of dynamically convex domains to the set of \zjnote{symplectically} convex domain\jznote{s} is unbounded.
\begin{thm} \label{prop-monotone-convex} $\sigma(\mathcal D_4, \mathcal C_4) = +\infty$. \end{thm}
\jznote{\noindent Along with the proof of Theorem \ref{prop-monotone-convex}, we discover a family of dynamically convex domains $X_{\Omega_p}$ (see (\ref{Xp})), parametrized by $p \in (0,1]$, that are not \jgnote{symplectomorphic} to convex ones when $p$ is sufficiently small. 

\medskip

\noindent {\bf Novelty:} The novelty of the toric domain $X_{\Omega_p}$ in the proof of Theorem \ref{prop-monotone-convex} consists of the following four points:}

\jznote{(i)  It is the first family of dynamically convex domains \jgnote{whose} symplectic non-convexity can be verified {\it without} referring to Chaidez-Edtmair's criterion.}

\jznote{(ii) The verification is only based on the classical machinery - ECH capacities (see \cite{Hut11}), where a concrete estimation on how small $p$ can be so that $X_{\Omega_p}$ is not \zjnote{symplectically convex} is obtained (see Remark \ref{numerical}). Note that an estimation on $p$ could also be obtained via Chaidez-Edtmair's criterion (See Remark \ref{Edt-rmk}). }

\zjnote{(iii) It shows a successful metric-geometrical application, based on a symplectic version of the classical John's ellipsoid theorem (see Section \ref{sec-main}) that quantitatively characterizes the property of being symplectomorphic to convex domains. In fact, our second main result Theorem \ref{thm-large-ratio} illustrates that this metric geometry approach is ``independent'' of Chaidez-Edtmair's criterion.} 

(iv) It indicates a new direction going to infinity in the large-scale geometry of the pseudo-metric space that consists of all the star-shaped domains in $\R^4$ equipped with $d_c$ (see Remark \ref{rmk-large-scale}). 

\zjrnote{\begin{remark} [About the point (ii) above] \label{Edt-rmk} Soon after the first version of this paper appears to the public, Oliver Edtmair informed us that, by chasing the arguments in \cite{CD-20}, one was able to obtain an explicit estimate of $C$ in the criterion (\ref{CE-ct}), though not presented in \cite{CD-20}. As a consequence, based on a careful calculation of $c_{\rm CE}(X_{\Omega_p})$, this upper bound also helps to estimate, in a relatively precisely manner, the parameter $p$ in $X_{\Omega_p}$ for symplectic non-convexity. 
\end{remark}}

\subsection{\jznote{$d_c$-large toric domains}} \label{ssec-td} 

\jznote{To state our first main result, towards the proof of Theorem \ref{prop-monotone-convex}, let us recall some notations from the symplectic toric geometry}. Let $\mu:\R^{2n}\to[0,\infty)$ be the standard moment map, i.e., 
\[\mu(z_1,\dots,z_n)=(\pi|z_1|^2,\dots,\pi|z_n|^2).\] A toric domain is a set of the form $X_\Omega=\mu^{-1}(\Omega)$, where $\Omega$ is the closure of a non-empty open set. We say that $X_\Omega$ is star-shaped if the origin is in the interior of $X_\Omega$ and $r \cdot X_\Omega\subset X_\Omega$ for $0<r<1$. Given a star-shaped toric domain $X_\Omega$, let $\partial_+\Omega$ denote the closure of $\partial\Omega\cap(0,\infty)^n$. In dimension 4, $\partial_+\Omega$ is a curve that connects the $x$ and the $y$ axes. We say that a star-shaped toric domain is \jznote{monotone} if $\partial_+\Omega$ is the graph of a decreasing function. \jznote{Let} $\mathcal{T}_4$ and $\mathcal{M}_4$ denote the sets consisting of the star-shaped toric domains and of the monotone toric domains, respectively. It was shown in \cite[Theorem 1.7]{HGR} that for every $U \in \mathcal M_4$, all the normalized symplectic capacities of $U$ coincide. The set $\mathcal M_4$ is quite a large set including all convex and concave toric domains, see \cite[Definition 2.1, 2.2]{HGR}, for example. In \cite[Proposition 1.8]{HGR}, it was shown that $\mathcal M_4=\mathcal T_4\cap \mathcal D_4$.

The following Figure \ref{fig_rel} illustrates the schematic relations between $\mathcal T_4$, $\mathcal C_4$, $\mathcal M_4$, and $\mathcal D_4$, \jznote{where $\mathcal M_4$ is the shaded region. }
\begin{figure}[h]
  \includegraphics[scale=0.65]{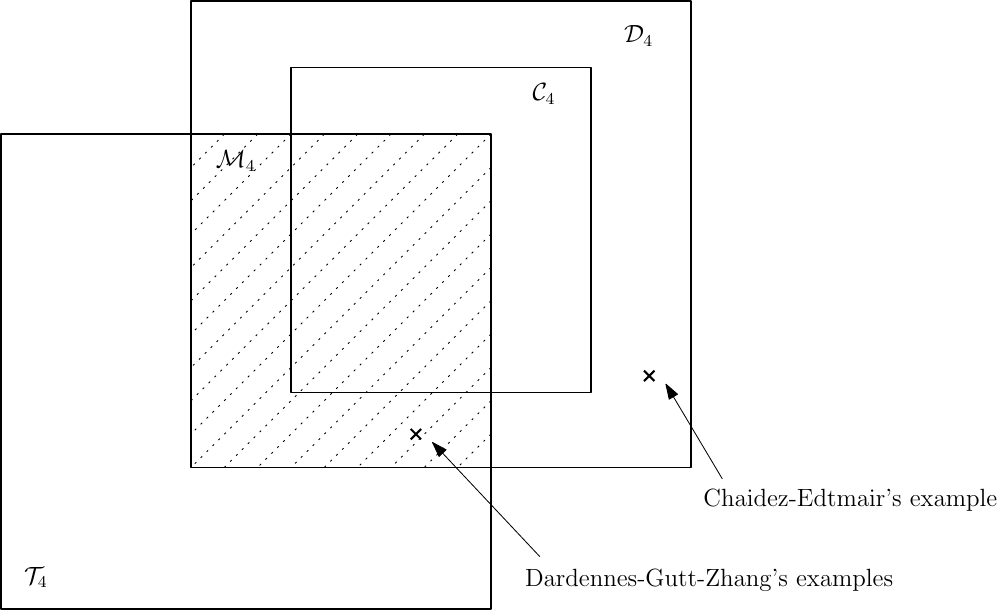}
  \caption{Relations between $\mathcal T_4$, $\mathcal C_4$, $\mathcal M_4$, and $\mathcal D_4$.}
  \label{fig_rel}
\end{figure}

As mentioned above, the relation $\mathcal C_4 \subset \mathcal D_4$ is given by \cite{HWZ}. It is strict due to examples that were recently found by \cite{CD-20} and \cite{DGZ}. So for the proof of Theorem \ref{prop-monotone-convex}, it suffices to find a family of monotone toric domains that are arbitrary far from all elements of $\mathcal C_4$. 

\jznote{Now consider the following toric domain denoted by $X_{\Omega_p}\in \mathcal{M}_4$, where Figure \ref{fig_p_ball_0} shows the moment image $\Omega_p$\zjnote{.}
\begin{figure}[h]
\includegraphics[scale=0.9]{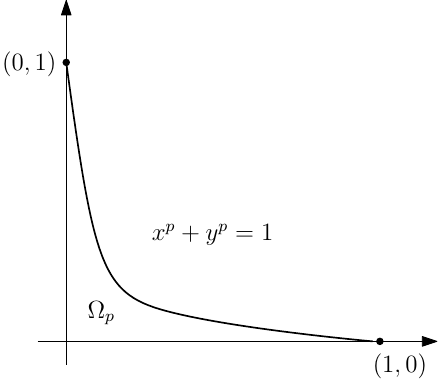}
\centering
\caption{$X_{\Omega_p}$ for $p \in (0,1]$.} \label{fig_p_ball_0}
\end{figure}} \jznote{Explicitly, for $p\in(0,1]$, let 
\begin{equation} \label{Xp}
X_{\Omega_p}:=\left\{(z_1,z_2)\in\C^2\mid \left(\pi|z_1|^2\right)^p+\left(\pi|z_2|^2\right)^p<1\right\}.
\end{equation}
}
\djnote{Define $d_c(X_{\Omega_p},\mathcal{E}_4)=\inf_{E \in \mathcal{E}_4} \,d_c(X_{\Omega_p},E)$, where $\mathcal{E}_4$ \zjnote{consists of} symplectic ellipsoids in $\R^4$ \zjnote{defined by (\ref{ellispoid})}}. Here is a main result.
\begin{thm} \label{ECH-est} For toric domain $X_{\Omega_p}$ defined as above, we have the following estimation for $p<\frac{1}{5}$,
\begin{equation}\label{eq:ineq}
d_c(X_{\Omega_p},\mathcal{E}_4)\ge \frac{1}{8} \log \left( \frac{g(p)}{1+\log 4+\log g(p)}\right)
\end{equation}
 where $g(p)=2^{\frac{2}{p}-2}{\rm{Vol}}_{\rm \R^4}(X_{\Omega_p})$. 
 In particular, when $p$ satisfies the condition that 
\begin{equation} \label{estimate}
\frac{g(p)}{1+ \log 4 + \log g(p)} \geq 256, 
\end{equation}
then $X_{\Omega_p}$ is dynamically convex but not symplectically convex. \end{thm}

\jznote{The proof of the second conclusion of Theorem \ref{ECH-est} directly comes from Proposition \ref{prop-symp-john}, while the first conclusion needs a sophisticated argument based on \zjnote{ECH capacities}, given in Section \ref{prop-monotone-convex}.} 
\begin{rmk}\label{numerical} \zjrnote{From the second conclusion of Theorem \ref{ECH-est}, one can estimate $p \in (0,1]$ so that $X_{\Omega_p}$ is symplectically non-convex.} \djnote{\zjnote{Explicitly, from the inductive relation that $g(\frac{1}{k+1})=\frac{2k+2}{2k+1}g(\frac{1}{k})$ for any positive integer $k$, one can show that for \zjnote{$p\le \frac{1}{62460059}$}, the condition (\ref{estimate}) is satisfied and such $X_{\Omega_p}$ is not symplectically convex.} \zjrnote{Here, we emphasize that little effort was made in the proof of Theorem \ref{ECH-est} to improve the estimation (\ref{estimate}), where the estimation from Chaidez-Edtmair's criterion (\ref{CE-ct}) mentioned in Remark \ref{Edt-rmk} potentially results in a better estimation on $p$.}  } 
\end{rmk}

\begin{rmk} For two domains $U, V \subset \R^4$, we call them {\rm symplectically equivalent} if $d_c(U, V) = 0$. Note that if $U$ is symplectomorphic to $V$, then they are symplectically equivalent by the definition of $d_c$. Formally, the symplectic equivalence relation is weaker (that is, more general) than being symplectomorphic. Back to $X_{\Omega_p}$, even though the criterion (\ref{CE-ct}) applies to verify its symplectic non-convexity, our Theorem \ref{ECH-est} in fact shows that they are not even symplectically equivalent to any convex domains when $p$ is sufficiently small. \end{rmk}

\jznote{\begin{proof}[Proof of Theorem \ref{prop-monotone-convex}] (Assuming Theorem \ref{ECH-est}) 

\djnote{In proof of Theorem \ref{ECH-est} in Section \ref{sec-main}, we show that $g(p)$ increases to $+\infty$ when $p$ goes to 0 so do the lower bound in (\ref{eq:ineq}). \zjnote{Then t}he desired conclusion results from the following triangular inequality,
$$d_c(X_{\Omega_p},\mathcal{E}_4)\le d_c(X_{\Omega_p},\mathcal{C}_4)+\sigma(\mathcal{C}_4,\mathcal{E}_4)$$
and Proposition \ref{prop-symp-john}.}
\end{proof}}

\begin{rmk} \label{rmk-large-scale} Consider the following pseudo-metric space
\[ \left(\mathcal S_4, d_c\right) = \left(\left\{\mbox{star-shaped domains in $\R^4$}\right\}, d_c\right). \]
From a perspective of the coarse geometry that focuses on large-scale geometrical phenomena, one could ask how many linearly independent directions in $(\mathcal S_4, d_c)$ that go to infinity, usually called the {\rm rank of a quasi-flat} in $(\mathcal S_4, d_c)$. The higher the rank is or the more directions there are, the richer star-shaped domains in $\R^4$ will be in terms of $d_c$. Immediately from the relation $\mathcal E_4 \subset \mathcal S_4$ together with Proposition \ref{prop-symp-john}, one concludes that there are at least 2 such directions, given by $\mathcal C_4$ consisting of all the symplectically convex domains. 

Moreover, the classical example, 1-finger shape from Hermann in \cite{Her}, shows that there exists a third direction. Here, by Theorem \ref{ECH-est} and a quick observation via the comparison of Lagrangian capacity $c_{\rm Lag}$ (see, for instance, \cite{CM18,Per22}) of $X_{\Omega_p}$ and Hermann's example, the family $X_{\Omega_p}$ for $p \in (0,1]$ indicates the existence of a new direction going to infinity in $(\mathcal S_4, d_c)$. 
\end{rmk}



\subsection{\zjnote{Differences in two criteria}} \label{ssec-comp}

%
%
%

As mentioned above, before this paper, the only technique that was \jgnote{known} for proving that dynamically convex domains are not \zjnote{symplectically convex was Chaidez--Edtmair's criterion} based on the Ruelle invariant \cite{hut_ruelle}. It turns out that this criterion and the \zjnote{metrical-geometric approach} used in this paper are quite different in essence as we now explain. To state the next result, let us introduce a notation: for any \djnote{star-shaped} domain $X\subset \R^4$ (where $\partial X$ is a contact manifold with the contact 1-form given by $\lambda = \lambda_{\rm std}|_{\partial X}$), define 
\begin{equation} \label{dfn-CE}
c_{\rm CE}(X) : =  \frac{{\rm Ru}(\partial X) \cdot \mbox{(min period of a Reeb orbit of $\partial X$)}}{{\rm Vol}(\partial X)},
\end{equation}
where ${\rm Ru}(\partial X)$ is the Ruelle invariant of $\partial X$ with respect to the contact 1-form $\lambda$ above (for more details, see \cite{hut_ruelle}) \zjnote{and ${\rm Vol}(\partial X)$ is the contact volume computed from the volume form $\lambda \wedge d\lambda$.} Then the criterion from \cite{CD-20} says that there exist constants $0< c\leq C$ such that if $X$ is symplectically convex, then one has the following numerical constraint, 
\begin{equation} \label{CE-ct}
c\leq c_{\rm CE}(X) \leq C.
\end{equation}

The following result shows that the Ruelle invariant based criterion (\ref{CE-ct}) and $d_c$ provide two \djnote{different} perspectives to study \djnote{symplectic non-convexity}. 
\begin{thm} \label{thm-large-ratio} There exist sequences of \djnote{star-shaped} toric domains $\{X_{\Omega_1^k}\}$ and $\{X_{\Omega_2^k}\}$ \jznote{in $\R^4$} such that 
\jznote{\[ \begin{aligned}\lim_{k\to \infty}\frac{c_{\rm CE}(X_{\Omega_1^k}) / c_{\rm CE}(X_{\Omega_1^1})}{d_c\left(X_{\Omega_1^k}, X_{\Omega_1^1}\right)} &=0,\\
\lim_{k\to \infty}\frac{c_{\rm CE}(X_{\Omega_2^k}) / c_{\rm CE}(X_{\Omega_2^1})}{d_c\left(X_{\Omega_2^k}, X_{\Omega_2^1}\right)}&=\infty.\end{aligned}\]}
\end{thm}

\begin{rmk} Since $d_c$ reflects the changes of symplectic capacities, Theorem \ref{thm-large-ratio} also implies that $c_{\rm CE}$ and symplectic capacities are independent to each other in general. It would be interesting to explore other numerical characterizations of toric domains (cf.~recent work from Hutchings \cite{Hutchings_private}). \end{rmk}

\zjrnote{\begin{rmk} [Informed by Oliver Edtmair] The non-toric examples constructed in \cite{CD-20} with small and large $c_{\rm CE}$ in fact lie in arbitrarily small neighborhoods of the round ball with respect to the coarse distance $d_c$. This shows another evidence of the independence between $c_{\rm CE}$ and $d_c$.  \end{rmk}}

\noindent{\bf Acknowledgement}. \jznote{The second author was partially supported by the ANR LabEx CIMI (grant ANR-11-LABX-0040) within the French State Programme ``Investissements d’Avenir''.  The first and second author were partially supported by the ANR COSY (ANR-21-CE40-0002) grant.} The third author was partially supported by NSF grant DMS-1926686, FAPERJ
grant JCNE E-26/201.399/2021 and a Serrapilheira Institute grant. The fourth author was supported by USTC Research Funds of the Double First-Class Initiative. We are grateful for communications with Michael Hutchings, particularly concerning the second family of examples appearing in Theorem \ref{thm-large-ratio}. We also thank the hospitality of the conference - Persistence Homology in Symplectic and Contact Topology - held in Albi, France and organized by the second author, which provides an opportunity to a close collaboration between authors. \zjrnote{Finally, we are grateful for Oliver Edtmair's crucial comments on the first version of this paper. }

\section{\jznote{Symplectic John's ellipsoid theorem}} \label{sec-main} 
The classical John's ellipsoid theorem \cite{John} says that convex domains in $\R^n$ are close to ellipsoids. Explicitly, denote by ${\rm Conv}(\R^n)$ the convex domains in $\R^n$. A quantitative comparison between any two convex domains inside ${\rm Conv}(\R^n)$ is given by the \zjnote{so-called} Banach-Mazur distance \cite{Rud}, 
\begin{equation} \label{dfn-bm}
d_{\rm BM}(U,V) : = \inf \left\{ \log \lambda \geq 0 \,\bigg| \, \begin{array}{ll} \mbox{$\exists A \in {\rm GL}(n)$, $u, v \in \R^n$ such that} \\ \frac{1}{\lambda}(U + u) \subset A(V + v) \subset \lambda(U + u) \end{array} \right\}
\end{equation}
where $\cdot +u$ and $\cdot +v$ denote translations while $\frac{1}{\lambda} \cdot$ and $\lambda \cdot$ stand for dilations. It is easily verified that $d_{\rm BM}$ defines a pseudo-metric on ${\rm Conv}(\R^n)$ and $d_{\rm BM}(U,V) = 0$ if and only if $U$ is an affine transformation of $V$.

Since $d_{\rm BM}$ is defined up to affine transformations \zjnote{in $\R^n$}, consider ${\rm Ell}(n): = \{\mbox{ellipsoids in $\R^n$}\}$. Then John's ellipsoid theorem \cite{John} says that for any $U \in {\rm Conv}(n)$, we have
\begin{equation} \label{thm-john}
d_{\rm BM} (U, {\rm Ell}(n)) : = \inf_{E \in {\rm Ell}(n)} \, d_{\rm BM}(U, E) \leq \frac{1}{2}\log n
\end{equation}
in particular, finite.

Recall that $\sigma = \sigma_{d_c}$, defined by (\ref{sigma}), is a symplectic version of the Banach-Mazur distance with respect to \zjnote{the coarse distance} $d_c$. Recall also that $\mathcal E_4$ consists of all the symplectic ellipsoids in $\R^4$. 


\begin{prop}[Symplectic John's ellipsoid theorem] \label{prop-symp-john} \jznote{$\sigma(\mathcal C_4, \mathcal E_4) \leq \log 2$}. 
\end{prop}

\begin{proof} Up to symplectomorphism, let $U \in \mathcal C_4$ be a convex domain of \zjnote{$\mathbb{R}^{4}$}, by the classical John's ellipsoid theorem (\ref{thm-john}), there exists an ellipsoid \zjnote{$E \subset \R^{4}$} such that 
\begin{equation} \label{John}
E \subset U \subset o + 4 \cdot (E-o)
\end{equation}
where $o$ is the center of $E$. Note that this $E$ is not necessarily a symplectic ellipsoid. However, by Williamson’s theorem on standard forms for symplectic ellipsoids, there exists a (linear) symplectomorphism $\phi \in \zjnote{{\rm Sp}(4)}$ such that $\phi(E) \in \mathcal E_4$. Without loss of generality, let us assume that $\phi(E) = E(a,b)$ \jznote{which is defined by 
\begin{equation}\label{ellispoid}
     E(a,b) :=\left\{(z_1, z_2) \in \C^2 \,\bigg| \, \frac{\pi |z_1|^2}{a} + \frac{\pi |z_2|^2}{b} \leq 1\right\}
\end{equation}
for some $0 <a \leq b$}. Then, since the shifting (by $o$) is also a symplectomorphism, \jznote{the relation} (\ref{John}) implies that $E(a,b)\, {\hookrightarrow} \,U \, {\hookrightarrow} \, 4E(a,b)$. By a rescaling, one gets  
\[ \frac{1}{2}E\left(4a, 4b\right) \,{\hookrightarrow} \,U\, {\hookrightarrow}\, 2 E\left(4a, 4b\right).\]
Therefore, \jznote{by the definition (\ref{dfn-dc})}, we have $d_c(U,\mathcal{E}_4)\leq \log 2$. \end{proof}

\begin{rmk} Let us clarify the notation of rescaling $\alpha E(a,b)$ for $\alpha \in \R_{>0}$ that appears in the proof of Proposition \ref{prop-symp-john}. Here is the definition, 
\[ \alpha E(a,b) := \left\{(\alpha z_1,\alpha z_2) \in \C^2 \,\bigg| \, \frac{\pi |z_1|^2}{a} + \frac{\pi |z_2|^2}{b} \leq 1\right\}\]
for $\alpha \in \R_{>0}$. In particular, we have an identification that $\alpha E(a,b) = E(\alpha^2 a, \alpha^2 b)$. This implies that $E(a,b) \hookrightarrow U$ if and only if $\frac{1}{\sqrt{\alpha}} E\left(\alpha a, \alpha b\right) \hookrightarrow U$. \end{rmk}

\begin{rmk} Proposition \ref{prop-symp-john} holds in any $2n$-dimensional case. More explicitly, we have $d_c(\mathcal C_{2n}, \mathcal E_{2n}) \leq \frac{1}{2} \log (2n)$.\end{rmk}

\begin{rmk} \label{rmk-evid-1} The monotonicity of $\sigma(-,-)$ implies that 
\jznote{\[ \sigma(\mathcal C_4, \mathcal M_4) \leq \log 2 \]}
since $\mathcal E_4 \subset \mathcal M_4$. However, for any symplectically convex domain $U \in \mathcal C_4$, it is not clear how to approximate $U$ via \jznote{{\rm non-ellipsoids}} domains $V \in \mathcal M_4$. Nevertheless, since $\mathcal M_4$ contains substantially more elements than $\mathcal E_4$, it is naturally to expect that the upper \djnote{bound} of \jznote{$\sigma(\mathcal C_4, \mathcal M_4)$} is much lower than $\log 2$. 

\jznote{Interestingly, this suggests a metric-geometrical approach to prove the strong Viterbo conjecture. Namely if $\sigma(\mathcal C_4, \mathcal M_4) =0$, then the strong Viterbo conjecture holds, due to \cite[Theorem 1.7]{HGR}. In fact, from the main result of a recent work \cite[Theorem 1.2]{C-GH23b}, the same conclusion holds \djnote{in higher dimensions} if $\sigma(\mathcal C_{2n}, \mathcal M_{2n}) =0$ \zjnote{for $n \geq 3$}.} 

\jznote{At last, let us emphasize that this metric-geometrical approach is sensitive to the distance \djnote{picked} to define $\sigma$. \zjnote{D}ue to communications with M.~Hutchings \cite{Hutchings_private}, if one replaces $d_c$ by a rather similar pseudo-metric called {\rm symplectic Banach-Mazur distance $d_{\rm SBM}$} that has been studied in \cite{GU,Usher-sbm,SZ}, then surprisingly $\sigma_{d_{\rm SBM}}(\mathcal C_4, \mathcal M_4)>0$.} 
\jgnote{This does not immediately \zjnote{disprove} the \zjnote{conjecture $\sigma(\mathcal C_4, \mathcal M_4) =0$} above as $\sigma_{d_{\rm SBM}}(-, -)\geq \sigma(\mathcal -,-)$.}\end{rmk}

\section{\jznote{Proof of Theorem \ref{ECH-est}}} 

Denoted by $c_k^{\rm ECH}$ the $k$-th ECH capacity for $k \in \N$. \jznote{For more background on ECH capacities, see \cite{Hut11}. Also, as explained in Section \ref{ssec-td}, it suffices to prove the first conclusion of Theorem \ref{ECH-est}.} \zjnote{For simplicity, the 4-dimensional volume ${\rm Vol}_{\R^4}(-)$ is denoted by ${\rm vol}(-)$.}

\begin{proof} 
Suppose that $E(a,b)\hookrightarrow X_{\Omega_p}\hookrightarrow \lambda E(a,b)$ for some $E(a,b)$ \zjnote{and $\lambda \geq 1$}.
We assume without loss of generality that $a\le b$. \djnote{Note that} $\Omega_p$ is a concave toric domain. For each $k \in \N$, let $x_k(p)$ denote the $x$-intercept of the line of slope $-\frac{1}{k}$, which is tangent to the curve $x^p+y^p=1$. It follows from a straight-forward computation that
\[x_k(p)=(1+k^{\frac{p}{p-1}})^{-\frac{1}{p}}+k(1+k^{\frac{p}{1-p}})^{-\frac{1}{p}}.\]
Consider a certain truncation of $\Omega_p$, defined by 
\[X_k(p)=\{(x,y)\in[0,\infty)^2\mid x^p+y^p\le 1\text{ and } \max(x,y)\le x_k(p)\}\]
which in shown in Figure \ref{fig_p_ball}. 
\begin{figure}[h]
\includegraphics[scale=0.95]{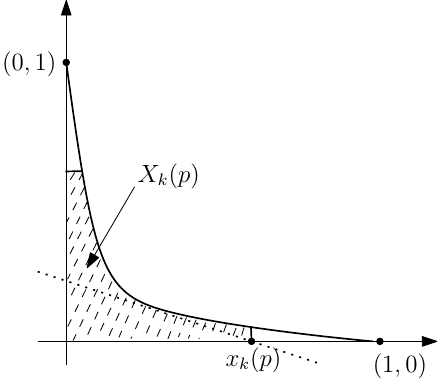}
\centering
\caption{\jznote{$X_k(p)$ for $p \in (0,1]$.}} \label{fig_p_ball}
\end{figure} 
Let $(w_1,w_2,\dots)$ be the weight decomposition of $\Omega_p$ (see \jznote{the corresponding} definition in \cite[Section 1.3]{Choi_2014}). Let $k=\lfloor b/a\rfloor$.
\begin{equation}\label{eq:1}
ka=c_k(E(a,b))\le c_k(X_{\Omega_p})=c_k\left(\bigsqcup_{i=1}^k B(w_i)\right).
\end{equation}
From \cite[Equation (3.9)]{hut_ruelle}, we obtain
\begin{equation}\label{eq:2}
c_k\left(\bigsqcup_{i=1}^k B(w_i)\right)\le 2\sqrt{k\cdot \text{vol}\left(\bigsqcup_{i=1}^k B(w_i)\right)}.
\end{equation}
It follows from the definition of the weight decomposition that $\bigsqcup_{i=1}^k B(w_i)\subset X_k(p).$
So
\begin{equation}
\begin{aligned}
\text{vol}\left(\bigsqcup_{i=1}^k B(w_i)\right)&\le\text{vol}(X_k(p))=2^{-\frac{2}{p}}+2\int_{2^{-\frac{1}{p}}}^{x_k(p)}(1-x^p)^{\frac{1}{p}}\,dx\\
&=2^{-\frac{2}{p}}\left(1+2\int_1^{2^{\frac{1}{p}}x_k(p)}(2-u^p)^{\frac{1}{p}}\,du\right).\label{eq:3}
\end{aligned}
\end{equation}
We note that the integrand in the last integral above increases as $p\to 0$ and it converges pointwise to $\frac{1}{u}$. Moreover it is readily verified that $2^{\frac{1}{p}}x_k(p)$ also increases as $p\to 0$ and it converges to $2\sqrt{k}$. So 
\begin{equation}\label{eq:4}
2\int_1^{2^{\frac{1}{p}}x_k(p)}(2-u^p)^{\frac{1}{p}}\,du\le 2\int_1^{2\sqrt{k}}\frac{1}{u}\,du=2\log\left(2\sqrt{k}\right)=\log 4+\log k.
\end{equation}
Combining \eqref{eq:1}, \eqref{eq:2}, \eqref{eq:3} and \eqref{eq:4}, we obtain $ka\le 2\cdot 2^{-\frac{1}{p}}\sqrt{k(1+\log 4+\log k)}$.
So
\begin{equation}\label{eq:5}
a\le 2\cdot 2^{-\frac{1}{p}}\sqrt{\frac{1+\log 4+\log k}{k}}.
\end{equation}
Since $X_{\Omega_p}\hookrightarrow\lambda E(a,b)$, we have
\begin{equation}\label{eq:6}
2\cdot 2^{-\frac{1}{p}}=c^{\rm ECH}_1(X_{{\Omega}_p})\le \lambda^2 c^{\rm ECH}_1(E(a,b))=\lambda^2 a.
\end{equation}
It follows from \eqref{eq:5} and \eqref{eq:6} that $a\le \lambda^2 a\sqrt{\frac{1+\log 4+\log k}{k}}$.
So
\begin{equation}\label{eq:7}
\lambda^4\ge \frac{k}{1+\log 4+\log k}.
\end{equation}
Note that this estimation is not enough to deduce the requested large $d_c$-distance conclusion since when $p$ changes, the constant $k$ (so that the corresponding ellipsoid $E(a,b)$ with $k=\lfloor b/a\rfloor$ that embeds into $X_{\Omega_p}$) may change. 
  
Now, from \eqref{eq:5}, we also obtain the following estimations, 
\begin{equation*}
\begin{aligned}
a&\le \sqrt{\frac{(1+\log 4+\log k)\text{vol}(X_{\Omega_p})}{k\cdot g(p)}}\\&\le \sqrt{\frac{(1+\log 4+\log k)\text{vol}(\lambda E(a,b))}{k\cdot g(p)}}\\
&=\sqrt{\frac{(1+\log 4+\log k)\lambda^4 ab}{2k\cdot g(p)}}\\
&\le a\sqrt{\frac{(1+\log 4+\log k)\lambda^4 (k+1)}{2k\cdot g(p)}}\\
&\le a\sqrt{\frac{(1+\log 4+\log k)\lambda^4}{g(p)}}
\end{aligned}
\end{equation*}
where the final step comes from the estimation $\frac{k+1}{2k} \leq 1$ for all $k \in \N$. So, we get 
\begin{equation}\label{eq:8}
\lambda^4\ge \frac{g(p)}{1+\log 4+\log k}.
\end{equation}
From \eqref{eq:7} and \eqref{eq:8} it follows that
\begin{equation}\label{eq:9}
\lambda^4\ge \frac{\max(g(p),k)}{1+\log 4+\log k}\zjnote{.}
\end{equation}
From a change of variables we obtain
\begin{equation*}
\begin{aligned}
g(p) = 2^{\frac{2}{p}-2}\text{vol}(X_{\Omega_p}) &=2^{\frac{2}{p}-2}\int_0^1(1-x^p)^{\frac{1}{p}}\,dx\\
&=\frac{1}{4}\int_0^{2^{\frac{1}{p}}}(2-u^p)^{\frac{1}{p}}\,du.
\end{aligned}
\end{equation*}
One observes that $\lim_{p\to 0}(2-u^p)^{\frac{1}{p}}=\frac{1}{u}$. Therefore, $g(p)$ increases as $p\to 0$ and $\lim_{p\to 0}g(p)=\infty$. Moreover, \[g\left(\frac{1}{5}\right)=\frac{\Gamma(1+5)^2}{\Gamma(1+2\cdot 5)}\cdot 2^{2\cdot 5-2}=\frac{5!^2}{10!}\cdot 2^{8}=\frac{64}{63}>1.\] So for $p<\frac{1}{5}$, it follows that $\lfloor g(p)\rfloor\ge 1$. In this case,
\djnote{
\begin{equation*}
\begin{aligned}
\min_k \frac{\max(g(p),k)}{1+\log 4+\log k}&=\min\left(\frac{g(p)}{1+\log 4+\log\lfloor g(p)\rfloor},\frac{\lfloor g(p)\rfloor +1}{1+\log 4+\log(\lfloor g(p)\rfloor+1)}\right)\\
& \ge \frac{g(p)}{1+\log 4+\log g(p)}.
\end{aligned}
\end{equation*}
}
Taking the infimum of over all triples $(\lambda,a,b)$, we obtain

\djnote{$$f(p)^4\ge \frac{g(p)}{1+\log 4+\log g(p)} $$
where $f(p)=\inf\{\lambda\ge 1\mid E(a,b)\hookrightarrow X_{\Omega_p}\hookrightarrow \lambda E(a,b)\text{ for some }a,b>0\}$.
Then one can show that $f(p)$ is linked to the coarse distance \zjnote{$d_c$} by 
$$d_c(X_{\Omega_p},\mathcal{E}_4)=\frac{1}{2} \log f(p)$$
}
thus we obtain \eqref{eq:ineq}, proving the desired conclusion. \end{proof}





Here, we \jznote{point out} that if we ``linearize'' the profile curve in $\Omega_p$ of $X_{\Omega_p}$ considered in the proof of Theorem \ref{ECH-est} above, denoted by $\Omega_p^{\rm lin}$ and shown in Figure \ref{fig_linear_p}, then following conclusion holds. 
\begin{figure}[h]
\includegraphics[scale=1]{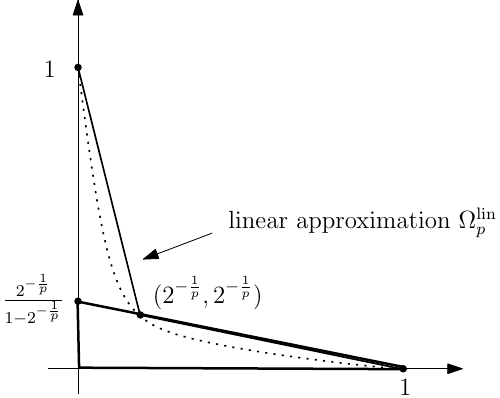}
\centering
\caption{A linear approximation of $X_{\Omega_p}$ from $\Omega_p^{\rm lin}$.} \label{fig_linear_p}
\end{figure}

\begin{prop} \label{prop-linear-approx} For any $p \in (0,1]$, we have $d_c\left(X_{\Omega_p^{\rm lin}}, \mathcal C_4\right) \leq \frac{1}{2} \log 3$. \end{prop}

\begin{proof} On the one hand, by inclusion, the ellipsoid $E(1, \frac{2^{-1/p}}{1-2^{-1/p}})$ shown by the bold edges in Figure \ref{fig_linear_p} embeds inside $X_{\Omega_p^{\rm lin}}$. On the other hand, by Lemma 4.5 in \cite{HGR}, $X_{\Omega_p^{\rm lin}} \hookrightarrow P(1, 2^{-1/p+1})$. Since $P(1, 2^{-1/p+1})\subset (3 - 2^{-1/p+1}) E(1, \frac{2^{-1/p}}{1-2^{-1/p}})$ trivially, one gets the following relation, 
    $$E\left(1, \frac{2^{-\frac{1}{p}}}{1-2^{-\frac{1}{p}}}\right)\hookrightarrow X_{\Omega_p^{\rm lin}} \hookrightarrow 3 E\left(1, \frac{2^{-\frac{1}{p}}}{1-2^{-\frac{1}{p}}}\right)$$
This implies that
    $$d_c\left(X_{\Omega_p^{\rm lin}} ,\mathcal{E}_4\right)\leq \frac{1}{2}\log 3$$
which completes the proof. \end{proof}

\begin{rmk} Here is \jznote{a} way to explain the essential difference between \jznote{Theorem \ref{prop-monotone-convex}} and Proposition \ref{prop-linear-approx}, by \jznote{simply} comparing \jgnote{the respective} volumes, where $$\frac{{\rm Vol}_{\R^{4}} (X_{\Omega_p^{\rm lin}})}{{\rm Vol_{\R^4}}(X_{\Omega_p})} \to +\infty$$ as $p \to 0$. \end{rmk}

\section{\jznote{Proof of Theorem \ref{thm-large-ratio}}} 

\begin{proof} The proof is given by two different examples. 

\medskip

\noindent {\bf Example One}: Let us consider the strangulation operation on $B^4(1)$ or for brevity $B(1)$. By definition, it cuts off a part of $B^4(1)$, which is, on the level of the moment image, a thin triangle symmetric to the diagonal $y = x$. See Figure \ref{fig_strangulation}, where $\ep(\delta)$ is proportional to $\delta$, 
\begin{figure}[h]
\includegraphics[scale=0.8]{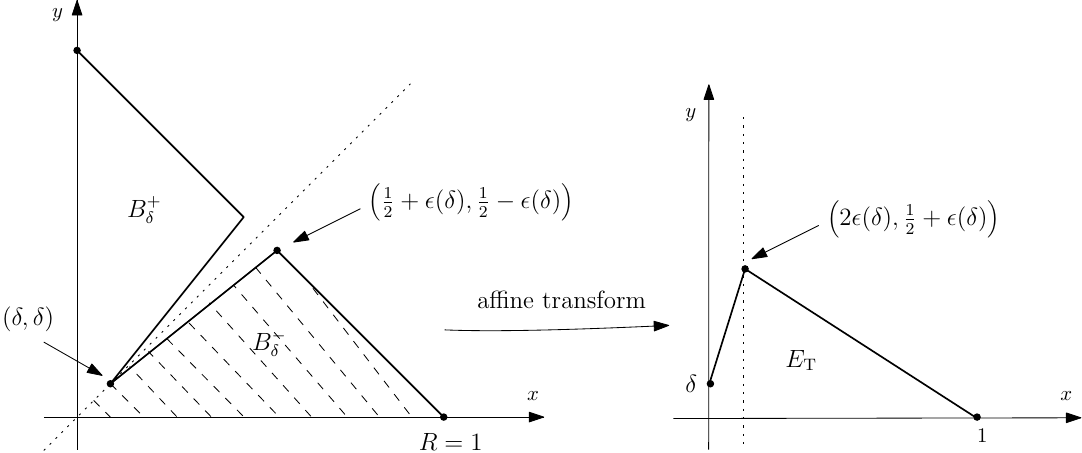}
\centering
\caption{A strangulation on $B(1)$, resulting a toric domain $B_{\delta}$.} \label{fig_strangulation}
\end{figure} 
Denote the resulting toric domain by $B_\delta$. Then obviously we have $B_\delta \hookrightarrow B(1)$ simply by inclusion. Now, cut through the diagonal and taking closures, we obtain two toric domains $B_\delta^-$ and $B_\delta^+$ (which are in fact symplectomorphic to each other by symmetry). Apply the affine transformation
\[ \begin{pmatrix}
1 & -1 \\
1 & 0 
\end{pmatrix},\] 
then the moment image of $B_{\delta}^-$ transfers to a region $E_{\rm T}$ where the corresponding toric domain (still denoted by) $E_{\rm T}$ is an example of so-called {\it truncated ellipsoid} defined in \cite{Usher-sbm}. More explicitly, by an elementary calculation, it is a truncated ellipsoid from $E\left(1, \frac{1 + 2\ep(\delta)}{2-4\ep(\delta)}\right)$, with truncated parameters $\delta$ and $\frac{\frac{1}{2} + \ep(\delta) - \delta}{2\ep(\delta)}$. These parameters are denoted by $\ep$ and $\beta$ respectively in \cite{Usher-sbm}. For our purpose, let use denote \begin{equation} \label{beta}
\beta(\delta) := \frac{\frac{1}{2} + \ep(\delta) - \delta}{2\ep(\delta)}.
\end{equation}
Since affine transformations on moment images induce symplectomorphisms on the corresponding toric domains, we know that 
\begin{equation} \label{affine}
(E_{\rm T})^- \sqcup (E_{\rm T})^+ \simeq B_{\delta}^- \sqcup B_{\delta}^+ \hookrightarrow B_\delta
\end{equation}
where $(E_{\rm T})^-$ and $(E_{\rm T})^+$ stand for two copies of the truncated ellipsoid above. One of the more striking properties of truncated ellipsoids is Corollary 5.3 in \cite{Usher-sbm}, which proves that 
\begin{equation} \label{Usher-small}
d_c\left(E_{\rm T}^-, E\left(1, \frac{1 + 2\ep(\delta)}{2-4\ep(\delta)}\right)\right) \leq 2 \log \left(\frac{1+\beta(\delta)}{\beta(\delta)}\right) : = \log C(\delta).
\end{equation}
Hence, with all the notations above, we have 
\begin{align*}
\frac{1}{C(\delta)} B\left(\frac{1 + 2\ep(\delta)}{2-4\ep(\delta)} \right) & \hookrightarrow \frac{1}{C(\delta)} B\left(\frac{1 + 2\ep(\delta)}{2-4\ep(\delta)} \right) \sqcup  \frac{1}{C(\delta)} B\left(\frac{1 + 2\ep(\delta)}{2-4\ep(\delta)} \right) \\
& \hookrightarrow \frac{1}{C(\delta)} E\left(1, \frac{1 + 2\ep(\delta)}{2-4\ep(\delta)}\right) \sqcup \frac{1}{C(\delta)} E\left(1, \frac{1 + 2\ep(\delta)}{2-4\ep(\delta)}\right)\\
& \hookrightarrow  (E_{\rm T})^- \sqcup  (E_{\rm T})^+ \hookrightarrow B_{\delta}
\end{align*}
where the first embedding is just the inclusion; the second embedding comes from the Gromov width of an ellipsoid; the third embedding comes from (\ref{Usher-small}); the last embedding is just (\ref{affine}). To sum up, we obtain the following embedding relations, 
\[ B\left(\frac{1}{C(\delta)^2} \cdot \frac{1 + 2\ep(\delta)}{2-4\ep(\delta)} \right) \hookrightarrow B_{\delta} \hookrightarrow B(1). \]
Then when $\delta \to 0$, we have 
\[ \ep(\delta) \to 0, \,\,\mbox{so} \,\, \frac{1 + 2\ep(\delta)}{2-4\ep(\delta)} \to \frac{1}{2}; \,\,\,\, \beta(\delta)  \to +\infty, \,\, \mbox{so} \,\, C(\delta) \to 1. \]
Therefore, $d_c(B(1), B_{\delta}) \to \log\sqrt{2}$ as $\delta \to 0$. In particular, $B(1)$ and $B_{\delta}$ do not differ much in terms of any symplectic capacity. \jznote{However, by the argument in \cite[Section 3.1]{DGZ}, we know $c_{\rm CE}(B_{\delta}) \to 0$ as $\delta \to 0$.}

\medskip

\noindent{{\bf Example Two}}: Let us consider the strain operation on $B^4(99)$ or for brevity $B(99)$. By definition, we add two small ``tail'' triangles along each axis on the moment image of $B(99)$. See Figure \ref{fig_strain}, \zjnote{where small triangles are shaded, they are symmetric with respect to $y = x$, and the horizontal one has height $1$ and base $99$}. Denote the resulting domain by $X$. 
\begin{figure}[h]
\includegraphics[scale=0.8]{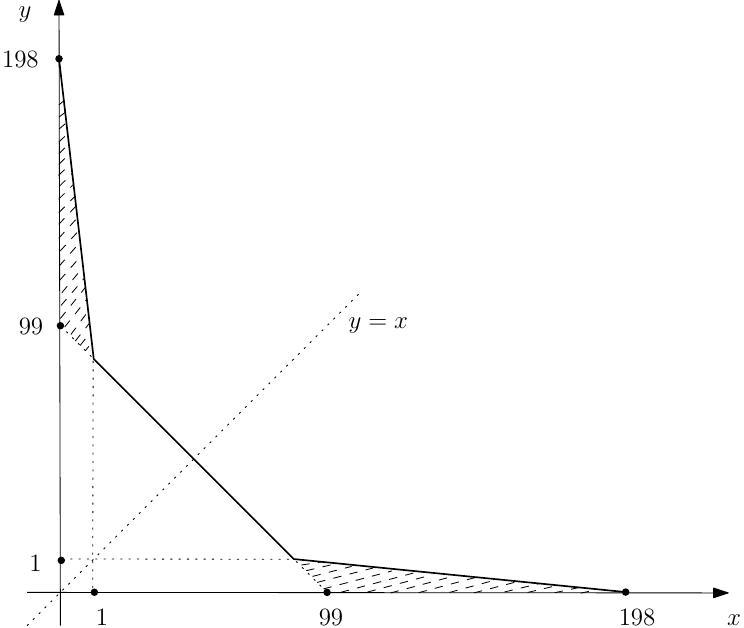}
\centering
\caption{A strain on $B(99)$, resulting a toric domain $X$.} \label{fig_strain}
\end{figure} 
Note that 
\[ {\rm Vol}_{\R^4}(X) = \frac{99^2}{2} + 2 \cdot \frac{99}{2} < \frac{100^2}{2} = {\rm Vol}_{\R^4}(B(100)). \]
Then obviously we have $B(99) \hookrightarrow X$. Now, we claim that $X \hookrightarrow B(100)$. \jznote{Since $X$ is concave toric, $B(100)$ is convex toric, and the weight sequence of $X$ is 
\[ (w_1, w_2, \cdots) = (99,  \underbrace{1, \cdots, 1}_\text{$198$-many} ) \]
such required embedding can be transferred to a ball-packing problem by \cite[Theorem 2.1]{Dan} . In fact, we claim that 
\[ B(99) \cup \underbrace{B(1) \cup \cdots \cup B(1)}_\text{$199$-many} \hookrightarrow B(100). \]
Indeed, the shaded tail triangles in $X$ (for instance, the horizontal one) can be divided into small triangles where each is equivalent to the triangle with vertices $(98,0), (99, 0), (98,1)$, up to affine transformation. For an illustration, see the left picture in Figure \ref{fig_packing}.  
\begin{figure}[h]
\includegraphics[scale=0.9]{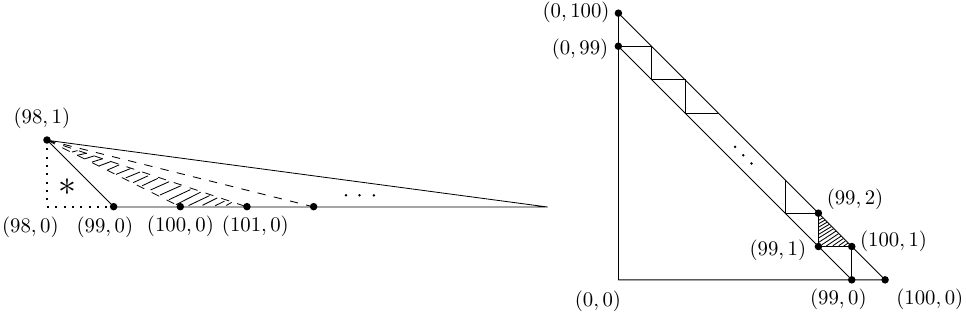}
\centering
\caption{Packing the tail triangles of $X$ into $B(100)$.} \label{fig_packing}
\end{figure} 
Then by further affine transformations, we can move the ``start'' triangle into $B(100) \backslash B(99)$, on the level of their moment image denoted by $\Delta(100) \backslash \Delta(99)$. The right picture in Figure \ref{fig_packing} shows how the small triangle, shaded one with vertices $(100, 0), (101, 0), (98,1)$, is placed \zjnote{inside} $\Delta(100) \backslash \Delta(99)$, precisely located at the shaded small triangle with vertices $(99,1), (100,1), (99,2)$. Since 
\[ {\rm Area}_{\R^2}(\Delta(100) \backslash \Delta(99)) = \frac{100^2}{2} - \frac{99^2}{2} = \frac{199}{2} = 199 \cdot {\rm Area}_{\R^2}(\mbox{small $\Delta$}),\]
such procedure can be inductively carried out and pack $199$-many small $\Delta$ into $\Delta(100) \backslash \Delta(99)$. This corresponds to a packing of $199$-many $B(1)$ into $B(100)\backslash B(99)$. Thus we \zjnote{obtain} the desired claim. }

\jznote{Note that the same argument works identically if we replace the tail triangle in $X$ by one with the furthest vertex being $(k, 0)$ (but with volume fixed). Denote the resulting toric domain by $X_k$, then we complete the proof by the argument in \cite[Section 3.2]{DGZ} where $c_{\rm CE}(X_k) \to \infty$ when $k \to \infty$. }\end{proof}

 \bibliographystyle{amsplain}
\bibliography{CCT}

\providecommand{\bysame}{\leavevmode\hbox to3em{\hrulefill}\thinspace}
\providecommand{\MR}{\relax\ifhmode\unskip\space\fi MR }
\providecommand{\MRhref}[2]{%
  \href{http://www.ams.org/mathscinet-getitem?mr=#1}{#2}
}
\providecommand{\href}[2]{#2}
\begin{thebibliography}{10}

\bibitem{AA-K-O-2014}
Shiri Artstein-Avidan, Roman Karasev, and Yaron Ostrover, \emph{From symplectic
  measurements to the {M}ahler conjecture}, Duke Math. J. \textbf{163} (2014),
  no.~11, 2003--2022. \MR{3263026}

\bibitem{CD-higher}
Julian Chaidez and Oliver Edtmair, \emph{Convexity and the {R}uelle invariant
  in higher dimensions}, arXiv preprint, arXiv: 2205.00935.

\bibitem{CD-20}
\bysame, \emph{3{D} convex contact forms and the {R}uelle invariant}, Invent.
  Math. \textbf{229} (2022), no.~1, 243--301. \MR{4438355}

\bibitem{Choi_2014}
Keon Choi, Michael Hutchings, Daniel Cristofaro-Gardiner, David Frenkel, and
  Vinicius Gripp~Barros Ramos, \emph{Symplectic embeddings into
  four-dimensional concave toric domains}, Journal of Topology \textbf{7}
  (2014), no.~4, 1054--1076.

\bibitem{CM18}
K.~Cieliebak and K.~Mohnke, \emph{Punctured holomorphic curves and {L}agrangian
  embeddings}, Invent. Math. \textbf{212} (2018), no.~1, 213--295. \MR{3773793}

\bibitem{Dan}
Dan Cristofaro-Gardiner, \emph{Symplectic embeddings from concave toric domains
  into convex ones}, Journal of Differential Geometry \textbf{112} (2019),
  no.~2, 199--232.

\bibitem{C-GH23b}
Dan Cristofaro-Gardiner and Richard Hind, \emph{On the agreement of symplectic
  capacities in high dimension}, arXiv preprint, arXiv: 2307.12125.

\bibitem{DGZ}
Julien Dardennes, Jean Gutt, and Jun Zhang, \emph{Symplectic non-convexity of
  toric domains}, Communications in Contemporary Mathematics,
  \url{https://doi.org/10.1142/S0219199723500104} (2023).

\bibitem{HGR}
Jean Gutt, Michael Hutchings, and Vinvius G.~B.~Ramos, \emph{Examples around
  the strong {V}iterbo conjecture}, J. Fixed Point Theory Appl. \textbf{24}
  (2022), no.~2, Paper No. 41, 22. \MR{4413022}

\bibitem{GU}
Jean Gutt and Michael Usher, \emph{Symplectically knotted codimension-zero
  embeddings of domains in {$\Bbb{R}^4$}}, Duke Math. J. \textbf{168} (2019),
  no.~12, 2299--2363. \MR{3999447}

\bibitem{Her}
David Hermann, \emph{Non-equivalence of symplectic capacities for open sets
  with restricted contact type boundary}, Can be found at
  \url{https://www.imo.universite-paris-saclay.fr/~biblio/pub/1998/abs/ppo1998_32.html}.

\bibitem{HWZ}
Helmut Hofer, Krzysztof Wysocki, and Eduard Zehnder, \emph{The dynamics on
  three-dimensional strictly convex energy surfaces}, Ann. of Math. (2)
  \textbf{148} (1998), no.~1, 197--289. \MR{1652928}

\bibitem{Hut11}
Michael Hutchings, \emph{Quantitative embedded contact homology}, J.
  Differential Geom. \textbf{88} (2011), no.~2, 231--266. \MR{2838266}

\bibitem{hut_ruelle}
\bysame, \emph{{E}{C}{H} capacities and the {R}uelle invariant}, J. Fixed Point
  Theory Appl. \textbf{24} (2022), no.~2, Paper No. 50, 25. \MR{4441526}

\bibitem{Hutchings_private}
\bysame, \emph{Private communications}, 2023.

\bibitem{John}
Fritz John, \emph{Extremum problems with inequalities as subsidiary
  conditions}, Studies and {E}ssays {P}resented to {R}. {C}ourant on his 60th
  {B}irthday, {J}anuary 8, 1948, Interscience Publishers, Inc., New York, 1948,
  pp.~187--204. \MR{0030135}

\bibitem{Per22}
Miguel Pereira, \emph{Equivariant symplectic homology, linearized contact
  homology and the lagrangian capacity}, PhD thesis, University of Augsburg,
  May 2022.

\bibitem{Rab78}
Paul~H. Rabinowitz, \emph{Periodic solutions of {H}amiltonian systems}, Comm.
  Pure Appl. Math. \textbf{31} (1978), no.~2, 157--184. \MR{467823}

\bibitem{Rud}
Mark Rudelson, \emph{Distances between non-symmetric convex bodies and the
  {$MM^\ast$}-estimate}, Positivity \textbf{4} (2000), no.~2, 161--178.
  \MR{1755679}

\bibitem{SZ}
Vuka\v{s}in Stojisavljevi\'{c} and Jun Zhang, \emph{Persistence modules,
  symplectic {B}anach-{M}azur distance and {R}iemannian metrics}, Internat. J.
  Math. \textbf{32} (2021), no.~7, Paper No. 2150040, 76. \MR{4284596}

\bibitem{Usher-sbm}
Michael Usher, \emph{Symplectic {B}anach-{M}azur distances between subsets of
  {$\Bbb{C}^n$}}, J. Topol. Anal. \textbf{14} (2022), no.~1, 231--286.
  \MR{4411106}

\bibitem{Viterbo2000}
Claude Viterbo, \emph{Metric and isoperimetric problems in symplectic
  geometry}, J. Amer. Math. Soc. \textbf{13} (2000), no.~2, 411--431.
  \MR{1750956}

\end{thebibliography}
\noindent\\

\end{document}